\documentclass[11pt]{amsart}

\usepackage{amssymb}
\usepackage{latexsym}
\usepackage{amsmath}
\usepackage{amsthm}
\usepackage{amsfonts}
\usepackage{color}
\usepackage{pdfsync}
\usepackage[usenames,dvipsnames]{pstricks}
\usepackage{epsfig}
\usepackage{pst-grad} 
\usepackage{pst-plot} 

\newcommand{\R}{\mathbb{R}}
\newcommand{\N}{\mathbb{N}}

\newcommand{\I}{\mathcal{I}}

\newcommand{\M}{\mathcal{M}}

\newcommand{\MM}{\mathcal{M}}

\def\R{{\mathbb R}}

\def\N{{\mathbb N}}

\def\cM{{\mathcal M}}

\def\cB{{\mathcal B}}

\def\k{\kappa}

\def\1{\left(}
\def\2{\right)}
\def\3{\left\{}
\def\4{\right\}}
\def\8{\infty}

\theoremstyle{plain}
\newtheorem{defi}{Definition}[section]

\newtheorem{teo}[defi]{Theorem}
\newtheorem{cor}[defi]{Corollary}
\newtheorem{lema}[defi]{Lemma}

\theoremstyle{definition}

\theoremstyle{remark}

\numberwithin{equation}{section}

\begin{document}

\title[]{Harnack inequality and its application to nonlocal eigenvalue problems in unbounded domains}
\author[]{Gonzalo D\'avila}
\address{
Departamento de Matem\'atica, Universidad T\'ecnica Federico Santa Mar\'ia \\
Casilla: v-110, Avda. Espa\~na 1680, Valpara\'iso, Chile
}
\email{gonzalo.davila@usm.cl}

\author[]{Alexander Quaas}
\address{
Departamento de Matem\'atica, Universidad T\'ecnica Federico Santa Mar\'ia \\
Casilla: v-110, Avda. Espa\~na 1680, Valpara\'iso, Chile}
\email{alexander.quaas@usm.cl}

\author[]{Erwin Topp}
\address{
Erwin Topp:
Departamento de Matem\'atica y C.C., Universidad de Santiago de Chile,
Casilla 307, Santiago, Chile.}
\email{erwin.topp@usach.cl}

\date{\today}

\begin{abstract}
We prove the Harnack inequality for general nonlocal elliptic equations with zero order terms.
As an application we prove the existence of the principal eigenvalue in general domains. Furthermore, we study the eigenvalue problem associated to the existence of self-similar solutions to the parabolic problem and provide estimates on the decay rate.
%
%
\end{abstract}
\maketitle
\section{Introduction.}

The study of the principal eigenvalue in unbounded domains, particularly in $\R^N$, appears naturally when studying the existence of self-similar solutions of parabolic equations. 

This can be seen for example in the classical heat equation
\begin{align*}
u_t-\Delta u=0\ \ \text{in }\R^N\times\R_+.
\end{align*}

Since the equation is invariant under the change $(x,t)\rightarrow (c^{1/2} x, c t)$ for $c>0$, the self-similar solutions $u$ must be of the form 
\begin{equation}\label{hola}
u(x,t)=\frac{1}{t^\lambda}v\left(\frac{x}{t^{1/2}}\right).
\end{equation}

After some simple computations we get that $v$ must solve 
\[
\Delta v+\frac{1}{2}x\cdot Dv=-\lambda v, \ \ \text{in } \R^N,
\]
which is an eigenvalue problem in the whole space. In this case one can search for radial solutions vanishing at infinity and get that the Gaussian profile 
$
v(x)=e^{-\frac{|x|^2}{4}}
$
solves the above eigenvalue problem with $\lambda= N/2$. 

The mentioned link among self-similar solutions for parabolic problems and eigenvalue problems in the Euclidean space has been addressed in the nonlinear setting as well. 
Consider a Lipschitz function
$F: \mathbb S^N \to \R$, positively $1$-homogeneous and  satisfying the uniform ellipticity condition
\begin{equation*}
\gamma \ \mathrm{Tr} (X - Y) \leq F(X) - F(Y) \leq \Gamma \ \mathrm{Tr}(X - Y), \quad \mbox{for all} \ X \geq Y,
\end{equation*}
for some $0 < \gamma \leq \Gamma < +\infty$ (these are the ellipticity constants).

In this setting,
Armstrong-Trokhimtchouk in \cite{AT}, and Meneses-Quaas in \cite{MQ} addressed the existence of an eigenpair $(\phi, \lambda) \in C(\R^N) \times \R_+$  with $\phi > 0$ of the problem
\begin{equation}\label{EqVPLo}
F(D^2\phi)+\frac{1}{2}x\cdot D \phi=-\lambda \phi \ \ \text{in }\R^N,
\end{equation}
extending to the fully nonlinear framework the notion of principal eigenvalue problems introduced in the classic work of Beresticky, Nirenberg and Varadhan in~\cite{BNV} (see also \cite{Hamel}).
In~\cite{AT, MQ}, the authors established Gaussian decay rates for properly normalized eigenfunction $\phi$ (say, $\phi(0) = 1$ or $\| \phi \|_\infty = 1$) of the form
\begin{align}\label{localdecay}
c e^{-\frac{|x|^2}{b}}\leq \phi \leq C e^{-\frac{|x|^2}{a}}, \quad x \in \R^N,
\end{align}
for some constants $c, C > 0$ and $0 < b \leq a$ just depending on the dimension $N$ and the ellipticity constants $\gamma, \Gamma$.


%


As in the classical case described above, eigenvalue problem~\eqref{EqVPLo} allows to obtain self-similar solutions for the fully nonlinear parabolic problem
\begin{align*}\label{EqParLo}
u_t-F(D^2u)=0 \quad \text{in }\R^N\times\R_+,
\end{align*}
and from here, qualitative properties of the solutions are obtained. 
We remark that, due to the nonlinear nature of the problem, there exists a second eigenpair $(\phi^-, \lambda^-) \in C(\R^N) \times \R_+$ solving~\eqref{EqVPLo}, with $\phi^- < 0$. 

\medskip

The previous discussion is the main motivation of this paper, where we aim to study of fractional principal eigenvalue problems in unbounded domains, and its subsequent relation to self-similar solution for fractional, fully nonlinear parabolic equations. 

We introduce the basic assumptions on the kernels defining the nonlocal operators. Let $s \in (0,1)$,  we say that $K: \R^N \to \R$ belongs to the class $\mathcal K_0$ if $K(y) = K(-y)$ for all $y$, and 
\begin{equation}\label{eee}
\frac{\gamma}{|y|^{N + 2s}} \leq K(y) \leq \frac{\Gamma}{|y|^{N + 2s}},
\end{equation} 
for some given constants $0 < \gamma \leq \Gamma < +\infty$. Given $K\in\mathcal K_0$, we denote
\begin{equation}\label{oplineal}
L(u,x) = L_K(u, x) = \mathrm{P.V.} \ \int_{\R^N} [u(y) - u(x)]K(x - y)dy.
\end{equation} 

For sets of indices $I, J$ (compact metric space), we consider a two parameter family of symmetric kernels $K_{ij}\in\mathcal K_0$ for all $i \in I, j \in J$, and introduce the nonlinear operator of Isaacs form
\begin{equation}\label{operador}
\I(u,x) = \inf_{i \in I} \ \sup_{j \in J} \ L_{ij} (u,x),
\end{equation} 
where $L_{ij} = L_{K_{ij}}$.





Our first main result is the following
\begin{teo}\label{teoex}
	Let $\Omega \subseteq \R^N$ (possibly unbounded), $s \in (1/2,1)$ and $\I$ a nonlocal operator given by~\eqref{operador}. Then for all 
	$b \in C(\Omega; \R^N)$, there exists an eigenpair $(\phi_1^+, \lambda_1^+)$ with $\phi_1^+ > 0$ in $\Omega$, solving the problem
	\begin{align}\label{eqintro}
	\I (\phi) + b \cdot D\phi = -\lambda \phi \ \text{in } \Omega; \qquad \phi = 0 \ \text{in } \Omega^c,
	\end{align}
	in the viscosity sense.
	
	The principal eigenvalue $\lambda_1^+ \geq 0$ is characterized by 
	\begin{align}\label{car}
	\lambda_1^+=\sup\{\lambda\left|\right.  \ \exists\phi>0 \ \text{in }\Omega,  \ \phi\geq0 \ \text{in }\R^N,  \ \text{s.t. } \I\phi + b \cdot D\phi \leq -\lambda \phi \ \text{in } \Omega\}.
	\end{align}
\end{teo}

An analogous existence result holds for the eigenpair associated to the negative eigenfunction. From now on we concentrate in the positive solution since all the results presented next can be written for the case of negative solution.

The proof of the previous theorem follows the ideas of \cite{Hamel, MQ}. The strategy is to consider a sequence of bounded domains $\Omega_n, \ n \in \N$ with smooth boundary such that $\Omega_n\nearrow\Omega$. In view of the results in~\cite{DQT}, there exists a sequence of solutions $(\phi_n, \lambda_n)$ of the eigenvalue problem~\eqref{eqintro} in $\Omega = \Omega_n$. Characterization~\eqref{car} holds for every $\Omega_n$ which leads to the boundedness of the family of eigenvalues $\lambda_n$. Elliptic regularity estimates lead to uniform bounds in $C^\alpha_{loc}(\Omega)$ for the family $\{ \phi_n \}$. The last ingredient to get a nontrivial limit for the sequence $\{ \phi_n \}$ is Harnack inequality for equations ``in resonance" form~\eqref{eqintro}. 

\medskip



Harnack inequality for elliptic nonlocal equations has been studied in different settings, from both the PDE and probability point of view, see for example (see for example \cite{Bass, Bass1, CS} ) and the references therein. We point out that all the previous references deal with operators without zero order terms. To the best of the author's knowledge the only known Harnack inequality for operators with zero order terms are for the special case of the fractional Laplacian, see \cite{Bogdan,CabreSire,CS1,Tan}. The spirit of the proof in \cite{CabreSire,CS1,Tan} relies on the extension work of \cite{CS1}, and the estimates for degenerate operators of Fabes, Kenig and Serapioni, see\cite{FKS}, while the work in \cite{Bogdan} is approached using probability tools. Note that the extension technique from \cite{CS1} does not apply for nonlinear operators.


The main result of this paper is the following.
\begin{teo}\label{Teoharnack}
	Let $\Omega \subset \R^N$ be a bounded domain, $s \in (1/2,1)$ and $\I$ a nonlocal operator given by~\eqref{operador}. Let $u\geq 0$ be a viscosity solution of
\begin{equation}\label{eqHarnack}
\I(u)+g(x,Du)+c(x)u=f(x), \ \ \text{in } \Omega,
\end{equation}
where $c,f \in L^\infty(\Omega)$ and $g$ satisfies
\[
|g(x,Du)|\leq M_1|Du|,
\]
for some $M_1$. Then for any $\tilde\Omega$ there exists a constant $C$ such that 
\[
\sup_{\tilde\Omega}u \leq C(\inf_ {\Omega}u+\|f\|_\infty).
\]
\end{teo}


As far as we know, this is the first Harnack inequality for fully nonlinear nonlocal operators with zero order terms, and we believe that this result is of independent interest, since it has several applications to different problems. The proof is based on the proof by Caffarelli and Silvestre for the fully nonlinear case without zero order terms, see \cite{CS}. The idea, as in the local case, is to consider an auxiliary function that solves an equation without the zero order term. In the local case this involves studying an equation with an extra transport term, meanwhile in the nonlocal case this translates into a nonlocal transport term, see Section \ref{SecHar} for more details.

\medskip

Coming back to the initial motivation about self-similarity in fractional parabolic problems, we briefly describe the known results for the linear 
fractional Heat equation, namely
\[
P_t - \Delta^sP=0 \ \ \text{in }\R^N\times\R_+,
\]
where $\Delta^s$ denotes the fractional Laplacian of order $2s$.

Fundamental solution associated to this problem has the self-similar form $P(x,t)=t^{-\frac{N}{2s}} \tilde \phi(|x|t^{-\frac{1}{2s}})$ and moreover one has the estimates 
\begin{equation}\label{cotafrac}
\frac{c_1 t}{(t^{1/s}+|x|^2))^{(N+2s)/2}}\leq P(x,t)\leq \frac{c_2 t}{(t^{1/s}+|x|^2))^{(N+2s)/2}}
\end{equation}
for some $c_2>c_1>0$, see  \cite{Blu, BPSV}. Therefore it is easily seen that the associated eigenpair $(\tilde \phi, \frac{N}{2s})$ solves the problem
\[
\Delta^s \tilde \phi+\frac{1}{2s}x\cdot D \tilde \phi=-\frac{N}{2s} \tilde \phi \quad \mbox{in} \ \R^N,
\]
and $\tilde \phi$ has polynomial decay of the form $|x|^{-N-2s}$ as $|x| \to \infty$. This is in striking difference with the exponential decay of the second-order version of the problem. This is a consequence of the decay at infinity of the kernel defining $\Delta^s$ (that is, the function $y \mapsto |y|^{-(N + 2s)}$), see Section~\ref{SecBarriers}.

We aim to get further properties of the eigenpair found in Theorem~\ref{teoex} in the case $\Omega = \R^N$ such as lower and upper bounds for the principal eigenvalue, and decay rates for the eigenfunction. We restrict ourselves 
to the case in which the nonlinear operator $\I$ is an extremal Pucci type operator, that is when $\I = \MM^-$ or $\I = \MM^+$ with
\begin{equation}\label{Pucci}
\MM^-(u,x) = \inf_{i \in I} L_i(u,x); \qquad \MM^+ (u,x) = \sup_{i \in I} L_i (u,x),
\end{equation}
where the inf and sup is taken over linear operators $L_i=L_{K_i}$ with $K_i\in\mathcal K_0$.

In what follows, singular fundamental solutions associated to the extremal operators $\MM^{\pm}$ play a key role. As it can be seen in~\cite{FQ}, there exist \textsl{dimension like numbers} $\tilde N^{\pm} > 0$ (depending on the ellipticity constants) such that the function $|x|^{2s - \tilde N^{\pm}}$ is a fundamental solution associated to $\M^{\pm}$, see~\eqref{fundamental}. It is known that $\tilde N^- \geq  N$ and therefore,  if $N \geq 2$, the fundamental solution of $\M^-$ is singular. On the other hand, inequality $\tilde N^+ > 2s$ requires further assumptions on $\gamma, \Gamma$ in order to hold. 

Note that in the linear case, and in particular the fractional Laplacian,  $\tilde N=N$ and so the the next result recover the bounds \eqref{cotafrac}.

\begin{teo}\label{mainTheo}
Assume $N \geq 2$. Let $s \in (1/2, 1)$, and $\M$ be either $\M^+$ or $\M^-$. Consider the eigenpair $(\phi_1^+, \lambda_1^+)$ solving 
\begin{align}\label{eqteo}
\M \phi +\frac{1}{2s}x\cdot D\phi=-\lambda \phi \ \ \text{in }\R^N,
\end{align}
with $\phi_1^+ > 0$ in $\R^N$. Assume that $\tilde N^+ > 2s$ in the case $\M = \M^+$.

Then, under the normalization $\phi_1^+(0) = 1$, there exist $c, C,\delta >0$ depending on the ellipticity constants, $N$ and $s$ such that  
\begin{align}\label{cotavp}
0<\frac{\tilde N-2s}{2s}\leq \lambda_1^+\leq \frac{N+2s-\delta}{2s},
\end{align}
and
\begin{align}\label{nondecay}
c|x|^{-(N+2s)}\leq \phi_1^+\leq C|x|^{-(N+2s)}, \quad \mbox{for} \ x \neq 0.
\end{align}

In addition, 
$\lambda_1^+$ is simple in the sense that if $(\phi, \lambda_1^+)$ is an eigenpair associated to the problem, then $\phi = t\phi_1^+$ for some $t \geq 0$; and unique, in the sense that if $(\lambda, \phi)$ solves~\eqref{eqteo} with $\phi > 0$, then $\lambda = \lambda_1^+$. As a consequence, $\phi_1^+$ is radially symmetric.
\end{teo}


Note that the bounds \eqref{nondecay} are the nonlocal equivalent of \eqref{localdecay}. To the best of the authors knowledge this is the first result regarding decay and bounds for the eigenvalue for nonlinear nonlocal operators. In what respects to~\eqref{nondecay}, this is obtained by carefully constructing global sub and super solutions for the eigenvalue problem. The construction needs two key ingredients: the refinement for fully nonlinear operators of the estimates found by Bonforte-V\'azquez in \cite{BV}; and an approximation argument for punctured domains in order to contruct barriers through the singular fundamental solution. In what respects to~\eqref{cotavp}, the estimates are related to the mentioned barriers and the characterization~\eqref{car}.

We finish the introduction with the following corollary, which is a direct consequence of the preceding theorem, see~\cite{AT, MQ}. 
\begin{cor}
Under the assumptions of Theorem~\ref{mainTheo}, $\lambda_1^+$ is the unique constant for which the parabolic problem
\begin{equation*}
u_t = \MM^+ u \quad \mbox{in} \ \R^N \times \R_+,
\end{equation*}
possess a self-similar solution $\Phi^+$ satisfying
\begin{equation*}
\Phi^+(x,t) = c^{\lambda_1^+} \Phi(c^{\frac{1}{2s}} x, ct) \quad \mbox{for all} \ x \in \R^N, \ t, c > 0.
\end{equation*}

This function is defined as $\Phi^+(x,t) = t^{-\lambda_1^+} \phi_1^+(x t^{-\frac{1}{2s}})$, where $(\phi_1^+, \lambda_1^+)$ is given in Theorem~\ref{mainTheo}. Moreover, $\Phi^+$ is unique up to a multiplicative positive constant.
\end{cor}


%
%

The paper is organized as follows. In Section \ref{SecPre} we define the class of operators we are working with and recall the notion of viscosity solution. In \ref{SecHar} we prove the Harnack inequality for fully nonlinear nonlocal operators with zero order terms. In Section \ref{SecVP} we prove the existence of a general eigenpair for general domains. Section \ref{SecBarriers} is dedicated to the construction of global sub and super solution for \eqref{eqteo}. Finally in Section \ref{SecDecay} we prove Theorem \ref{mainTheo}.

\medskip

\section{Preliminaries, notation and notion of solution}\label{SecPre}

We recall the definition of linear nonlocal operators $L$ as in~\eqref{oplineal}, nonlinear nonlocal operators $\I$ of Isaacs form as in~\eqref{operador} and the maximal operators in~\eqref{Pucci}.

For a measurable set $A \subseteq \R^N$, we denote
\begin{equation*}
L[A](u,x) = \int_{A} [u(y) - u(x)]K(x-y)dy,
\end{equation*}
and similar notation for nonlinear operators.

For $K\in\mathcal K_0$ and $u, v$ regular, bounded functions, we denote
\[
\cB_K(u,v)(x) =\frac{1}{2}\int_{\R^N}(u(x)-u(y))(v(x)-v(y))K(y-x)dy,
\]
and its associated extremal operators
\[
\cB^+(u,v)=\sup_K \cB_K(u,v) \quad \mbox{and} \quad \cB^-(u,v)=\sup_K \cB_K(u,v), 
\] 
This operator plays the role of the gradient, more precisely for smooth functions $u$ and $v$ it is easy to verify that $\cB(u,v)\to Du\cdot Dv$ and appears naturally in Section \ref{SecHar}. 

For $b:\R^N\to\R^N$ and $\kappa, \ c:\R^N\to\R$ bounded functions, we denote the Pucci extremal operators for this class
\[
\cM_{\cB}^{+}u(x)= \cM ^+u(x)+\kappa(x)\cB^+(u,\zeta)+b(x)\cdot Du+c(x)u, 
\]

In an analogous way we define $\cM_{\cB}^{-}$.

We say that $u: \R^N \to \R$ upper semicontinuous is a viscosity subsolution to the problem
$$
\I (u) + b \cdot Du + cu = -\lambda u, 
$$ 
at a point $x_0 \in \R^N$ if for each $\varphi \in C^2(B_\rho(x_0)) \cap L^\infty(\R^N)$ such that $x_0$ is a maximum point for $u - \varphi$ in $B_\rho(x_0)$, then 
\begin{equation*}
\I[B_\rho(x_0)](\varphi, x_0) + \I[B_\rho(x_0)^c](u, x_0) + b(x_0) \cdot D\varphi(x_0) + c(x_0) u(x_0) \geq -\lambda u(x_0).
\end{equation*}

Viscosity supersolution is defined in an analogous way.

\section{Harnack Inequality}\label{SecHar}

A key ingredient in the proof of the Harnack inequality is the $L^\varepsilon$ Lemma.
\begin{lema}\label{epsilonlemma}
Let $u\geq 0$ in $\R^N$ and suppose that $\cM_{\cB}^- u \leq C_0$ in $B_{2r}$ in the viscosity sense. Assume that $s>1/2$, then there exists a universal constant $C$ such that
\begin{align}\label{epsilonestimate}
|\{u>t\}\cap B_r|\leq Cr^n(u(0)+C_0r^{2s})^\varepsilon t^{-\varepsilon}.
\end{align}
\end{lema}
Note that when $c\geq 0$ the previous lemma is a direct adaptation of the one proven in \cite{CD}, since $\cB$ can be bounded by an extremal operator of order $\tau=2s-1$ and the local gradient can be absorbed by scaling. 

The general case also follows by rescaling in the ABP estimates and the first step of the proof of the $L^\varepsilon$.

Next we state and prove a Harnack inequality for operators with general lower order terms.

\begin{teo}\label{Teohar}
Let $u\geq 0$ be a continuous function in $\R^N$ and suppose that 
\[
\cM^-u(x)-M_1|Du|-M_2u\leq C_0,  \ \ \cM^+u(x)+M_1|Du|+M_2u\geq -C_0 \ \ \text{in } B_2,
\]
in the viscosity sense. Assume that $s>1/2$, then there exists a universal constant $C$ such that 
\[
u(x)\leq C(u(0)+C_0) \ \ x\in B_{1/2}.
\]
\end{teo}

The proof of Theorem \ref{Teoharnack} is a direct consequence of Theorem \ref{Teohar}.

\begin{proof}[Proof of Theorem \ref{Teohar}]

First consider the eigenpair $(\varphi ,\lambda_1)$ solving
\begin{align*}
\cM^+\varphi+ M_1 |D\varphi|(x)&=-\lambda_1 \varphi(x)\ \ \text{in } B_{2l},\\
\varphi&=0, \ \ \text{in } B^c_{2l}.
\end{align*}
where $\cM$ denotes the extremal operator (see \eqref{Pucci}). It is direct to check that $\varphi\in C^{2s+\alpha}(B_{2l})$ for some $\alpha>0$ (see for example \cite{Serra}).

A standard application ABP yields that $\lambda_1\to\infty$ as $l\to 0$. With this in mind let $l$ small so that 
\[
\lambda_1> 2 M_2.
\]
Also note that $\varphi>0$ in $B_l$ and we can also assume $\varphi(x)\geq 1$ in $B_{l}$. Consider now $\delta>0$ small to be fixed later and define $\varphi_\delta=\varphi+\delta$. 

 For simplicity of exposition we will assume that $l=1$. Note that all our computations will be done in $B_1$, where $\varphi(x)\geq 1$. The general case follows the same computations with every ball now rescaled by $l$ (i.e. $B_{\theta r}$ is replaced by $B_{\theta r l}$). Finally the result will hold within a ball of radius $B_{l/4}$ which then by a classical argument can be extended to the whole domain.

We proceed as in \cite{CS}. Without loss of generality we can assume that $u(0)\leq 1$ and $C_0=1$. Let now $\varepsilon$ be the one given by Lemma \ref{epsilonlemma} and let $\gamma=n/\varepsilon$. Consider $t$ the minimum value such that 
\[
\frac{u(x)}{\varphi_\delta}\leq h_t(x)=t(1-|x|)^\gamma, \ \ \text{for all }x\in B_1.
\]
Denote $x_0$ the point where equality holds, that is, $u(x_0)=\varphi_\delta(x_0)h_t(x_0)$, and set $d=(1-|x_0|)$, $r=d/2$. Note that with this notation we have 
\[
u(x_0)=\varphi(x_0)d^{-\gamma}.
\]
Our goal is to prove that $t$ is bounded by a universal constant, which then implies that $u$ is uniformly bounded in $B_{1/2}$.

Consider the set $A=\{u>u(x_0)/2\}$, by the $L^\varepsilon$ lemma \ref{epsilonlemma} (with $\zeta=0$) we get that there exists a universal constant $C$ such that 
\[
|A\cap B_1|\leq C\left|\frac{2}{u(x_0)}\right|\leq Ct^{-\varepsilon}d^n,
\]
where the constant $C$ on the second inequality depends on the $L^\infty$ norm of $\varphi$. Since $|B_r|=Cd^n$, we deduce the following estimate
\begin{align}\label{contradiction1}
|A\cap B_r(x_0)|\leq C t^{-\varepsilon}|B_r|.
\end{align}
Now we need to estimate $A^c\cap B_r(x_0)$. We will arrive to a contradiction if we are able to prove that there exists a positive constant $\mu$ independent of $t$ such that $A^c\cap B_r(x_0)\leq (1-\mu)|B_r|$.

Let $\theta>0$ be small so that $(1-\theta/2)^{-\gamma}\sim 1$  and consider 
\[
v(x)=\left(1-\frac{\theta}{2}\right)^{-\gamma}\frac{u(x_0)}{\varphi_\delta(x_0)}-\frac{u(x)}{\varphi_\delta(x)}.
\]
Note that for $x\in B_{\theta r}(x_0)$ we have 
\[
\frac{u(x)}{\varphi_\delta(x)}\leq h_t(x)\leq t\left(\frac{d-\theta d}{2}\right)^{-\gamma}\leq \frac{u(x_0)}{\varphi_\delta(x_0)}\left(1-\frac{\theta}{2}\right)^{-\gamma}.
\]
With this we conclude that $v\geq 0$ in $B_{\theta r}$. We would like to apply the $L^\varepsilon$ lemma to $v$, but since $v$ is not positive we need to truncate it first. Define then $w(x)=v^+(x)$, where $v=v^+-v^-$, and let us estimate $\cM^-w$ in $B_{\theta r /2}$. We have
\begin{align}\label{eqnw}
\cM^-w\leq \cM^-v+\cM^+(v^-).
\end{align}

Let us analyze $\cM^-v$. First, note that 
\[
\cM^-v=\cM^-(-u/\varphi_\delta)=-\cM^+(u/\varphi_\delta)
 \]
 

Since $\varphi_\delta$ is a smooth strictly positive function then the $u/\varphi_\delta$ satisfies the product rule in the viscosity sense. More precisely, let $\phi$ be a test function touching $V=u/\varphi_\delta$ from above at $\bar x\in B_1$, 
\[
\phi (\bar x)=\frac{u(\bar x)}{\varphi_\delta(\bar x)}, \ \ \phi(x)\geq \frac{u(x)}{\varphi_\delta(x)}, \text{in } B_{\rho}(\bar x)
\]
then the $\varphi_\delta\phi$ is a test function for $u$ (recall $\varphi_\delta >0$ and $\cM^+ u\geq -1$) hence we get, from testing, 
\begin{equation}\label{testH}
\begin{split}
\cM^+[B_\rho(\bar x)](\varphi_\delta\phi)(\bar x)+\cM^+[B^c_\rho(\bar x)](u)(\bar x) & \\
+M_1|D(\varphi_\delta\phi)|(\bar x))+M_2\varphi_\delta\phi(\bar x) & \geq -1.
\end{split}
\end{equation}
Now observe that since $\varphi_\delta$ and $\phi$ are smooth we have
\begin{align*}
\cM^+[B_\rho(\bar x)](\varphi_\delta(\bar x)\phi(\bar x))\leq & \ \phi(\bar x)\cM^+[B_\rho(\bar x)](\varphi_\delta)(\bar x)+\varphi_\delta(\bar x)\cM^+[B_\rho(\bar x)](\phi)(\bar x) \\
&+ 2\cB^+[B_\rho(\bar x)](\varphi_\delta,\phi)(\bar x).
\end{align*}
On the other hand, since $u=V\varphi_\delta$ and the kernels are non singular in $B_\rho^c(\bar x)$, we obtain the analogous  product rule for the term $\cM^+[B^c_\rho(\bar x)](u)(\bar x)$, that is
\begin{align*}
\cM^+[B^c_\rho(\bar x)](u)(\bar x)= & \ \cM^+[B^c_\rho(\bar x)](V\varphi_\delta)(\bar x)\\
\leq & \ V(\bar x)\cM^+[B^c_\rho(\bar x)](\varphi_\delta)(\bar x)+\varphi_\delta(\bar x)\cM^+[B^c_\rho(\bar x)](V)(\bar x) \\
& + 2\cB^+[B^c_\rho(\bar x)](\varphi_\delta,V)(\bar x)
\end{align*}
Adding both inequalities and noting that $V(\bar x)=\phi(\bar x)$ we get 
\begin{equation}\label{ineqH} 
\begin{split}
& \cM^+[B_\rho(\bar x)](\varphi_\delta\phi)(\bar x)+\cM^+[B^c_\rho(\bar x)](u)(\bar x) \\
\leq & \ \phi(\bar x)\cM^+\varphi_\delta(\bar x) +\varphi_\delta(\bar x) \Big{(} \cM^+[B_\rho(\bar x)](\phi)(\bar x) +\cM^+[B^c_\rho(\bar x)](V)(\bar x) \Big{)}\\
&+2\Big{(} \cB^+[B_\rho(\bar x)](\varphi_\delta,\phi)(\bar x)+\cB^+[B^c_\rho(\bar x)](\varphi_\delta,V)(\bar x) \Big{)}
\end{split}
\end{equation}

Observe that the last two terms represent the viscous testing of $V$ by $\phi$ for the functional
$\varphi_\delta\cM^+(\cdot)+2\cB(\cdot,\varphi_\delta)$. Now, since $D\varphi_\delta\phi=\varphi_\delta D\phi+\phi D \varphi_\delta$ we get from \eqref{testH}
\begin{align}\label{testH2}
\begin{split}
& \cM^+[B_\rho(\bar x)](\varphi_\delta\phi)(\bar x)+\cM^+[B^c_\rho(\bar x)](u)(\bar x) \\
\geq & \ -1-M_1\Big{(} \varphi_\delta (\bar x)|D\phi|(\bar x)+\phi (\bar x)|D \varphi_\delta| (\bar x)) \Big{)} -M_2\varphi_\delta(\bar x)\phi(\bar x)
\end{split}
\end{align}


Furthermore we have
\begin{align}\label{ineqH2}
\phi(\bar x)\cM^+\varphi_\delta(\bar x)&=\phi(\bar x)(-M_1|D\varphi_\delta| (\bar x)-\lambda_1\varphi_\delta(\bar x)+\lambda_1\delta)
\end{align}

Combining inequalities \eqref{testH2}, \eqref{ineqH}, \eqref{testH2} and dividing by $\varphi_\delta(\bar x)$ we deduce
\begin{align*}
 \frac{\phi(\bar x)}{\varphi_\delta(\bar x)}(-\lambda_1\varphi_\delta(\bar x)+\lambda_1\delta)+H(V,\phi,\varphi_\delta) \geq \frac{-1}{\varphi_\delta(\bar x)}-M_1 |D\phi|(\bar x)-M_2\phi(\bar x),
\end{align*}
where 
\begin{align*}
H(V,\phi,\varphi_\delta):=&\cM^+[B_\rho(\bar x)](\phi)(\bar x) +\cM^+[B^c_\rho(\bar x)](V)(\bar x)\\
&+\frac{1}{\varphi_\delta(\bar x)}2(\cB^+[B_\rho(\bar x)](\varphi_\delta,\phi)(\bar x)+\cB^+[B^c_\rho(\bar x)](\varphi_\delta,V)(\bar x))
\end{align*}
Notice that since $\lambda_1\geq 2M_2$ we can pick $\delta$ small enough (independent of $t$ and $u$ such that
\[
M_2-\lambda_1+\frac{\lambda_1\delta}{\varphi_\delta(\bar x)}<0
\]
since $\bar x\in B_1$ and $\varphi_\delta(x)\geq 1$ in $B_1$. We deduce then
\[
H(\phi,\varphi_\delta,v)+M_1|D\phi|(\bar x)\geq -1,
\]

and so we get that the function $v=-u/\varphi_\delta$ satisfies
\[
\cM^-v-\frac{2}{\varphi_\delta( x)}\cB^+(\varphi_\delta,v)-	M_1 |Dv|\leq 1,\ \ \text{in } B_1,
\]
in the viscosity sense.

Going back to \eqref{eqnw} we get
\begin{align*}
\cM^-w&\leq \cM^-v+\cM^+(-v^-)\\
&\leq 1 +\frac{2}{\varphi_\delta(\bar x)}\cB^+(\varphi_\delta,v)+M_1|Dv| +\cM_{\cB}^-(-v^-)
\end{align*}
Note that $\cB(\varphi_\delta,w)=\cB(\varphi_\delta,w+C)$ for any constant and that $v\geq 0$ in $B_{\theta r}$, furthermore
\[
\cB^+(u_1+u_2,h)\geq \cB^-(u_1,h)+\cB^+(u_2,h).
\]
Decompose again $v=v^+-v^-$ to get
\begin{align}\label{eqnw2}
\cM^-w-\frac{2}{\varphi_\delta(x)}\cB^-(\varphi_\delta,w)-M_1|Dw|\leq 1 +\cM^-(-v^-)-\cB^+(\varphi_\delta, -v^-).
\end{align}
We need to estimate $\cM^-(v^-)$ and $\cB^+(\varphi_\delta, v^-)$. The estimates are fairly similar as the ones in \cite{CS}. Let $x\in B_{\theta r/2}(x_0)$ and $C_\theta=(1-\theta/2)^{-\gamma}$, we have 
\begin{align}
\nonumber \cM^-(v^-)&\leq C(1-s) \int_{\{v(x+y)< 0\}}-\Lambda \frac{v(x+y)}{|y|^{n+2s}}\\
\nonumber  &\leq C(1-s)\int_{\R^N\setminus B_{\theta r/2}(x_0-x)}\frac{1}{|y|^{n+2s}}\left(\frac{u(x+y)}{\varphi_\delta(x+y)}-C_\theta\frac{u(x_0)}{\varphi_\delta(x_0)}\right)^+dy\\
\nonumber &\leq  (1-s)\int_{\R^N\setminus B_{\theta r/2}(x_0-x)}\frac{1}{|y|^{n+2s}}\left(\frac{u(x+y)}{\delta}-C_\theta\frac{u(x_0)}{\varphi_\delta(x_0)}\right)^+dy\\
\label{Mv-} &\leq  \frac{(1-s)}{\delta}\int_{\R^N\setminus B_{\theta r/2}(x_0-x)}\frac{1}{|y|^{n+2s}}\left(u(x+y)-C_\theta u(x_0)\delta\right)^+dy
\end{align}
To estimate the integral in the last inequality consider the function $g_\tau=\tau(1-|4x|^2)$ and pick the largest value of $\tau$ such that the inequality $u\geq g_\tau$. Since $u$ is positive in $\R^N$ there exists a point $x_1$ in $B_{1/4}$ satisfying $u(x_1)=g_\tau(x_1)$. Note that $\tau\leq 1$ since $u(0)\leq 1$. We can bound then 
\[
(1-s)\text{P.V.}\int_{\R^N}\frac{(u(x_1+y)-u(x_1))^-}{|y|^{n+2s}}dy\hspace{50mm}
\]
\[
\quad\quad\leq(1-s)\text{P.V.}\int_{\R^N}\frac{(g_\tau(x_1+y)-g_\tau(x_1))^-}{|y|^{n+2s}}dy\leq C,
\]
for $C$ independent of $s$. At $x_1$ we have $Du(x_1)=Dg_\tau (x_1)$, hence we deduce from 
\[
\cM^-u(x_1)-M_1|Du|(x_1)-M_2u(x_1)\leq 1,
\]
the inequality
\[
(1-s)\text{P.V.}\int_{\R^N}\frac{(u(x_1+y)-u(x_1))^+}{|y|^{n+2s}}dy\leq C,
\]
for a universal constant $C$ independent of $s$. From here, since $u(x_1)=g_\tau(x_1)\leq 1$ and $u\geq 0$ we get 
\[
(1-s)\int_{\R^N}\frac{(u(x_1+y)-2)^+}{|y|^{n+2s}}dy\leq C.
\]
Without loss of generality we can assume that $u(x_0)> 2\delta^{-1}$, since otherwise $t$ would be uniformly bounded (recall $\delta$ is small but independent of $t$ and $u$). With this we can bound the integral appearing in \eqref{Mv-}.
\begin{align*}
&\frac{(1-s)}{\delta}\int_{\R^N\setminus B_{\theta r/2}(x_0-x)}\frac{1}{|y|^{n+2s}}\left(u(x+y)-C_\theta u(x_0)\delta\right)^+dy\\
&=\frac{(1-s)}{\delta} \int_{\R^N\setminus B_{\theta r/2}(x_0-x)}\frac{(u(x_1+y+x-x_1)-C_\theta u(x_0)\delta)^+}{|y+x_1-x_0|^{n+2s}}\frac{|y+x_1-x_0|^{n+2s}}{|y|^{n+2s}}dy\\
&\leq \frac{C}{\delta}(\theta r)^{-n-2s}(1-s)\text{P.V.}\int_{\R^N}\frac{(u(x_1+y)-2)^+}{|y|^{n+2s}}dy\\
&\leq \frac{C}{\delta}(\theta r)^{-n-2s}.
\end{align*}
Observe that this estimate is the same as the one appearing in \cite{CS}, though we assume that $u(x_0)> 2\delta^{-1}$, instead of $u(x_0)>2$.

We need to estimate now $-\cB^+(\varphi_\delta,- v^-)=\cB^-(\varphi_\delta, v^-)$. For this is enough to notice that 
\begin{align*}
-\cB^+(\varphi_\delta, v^-)&\leq (1-s)\Lambda^{-1}\int_{\R^N}\frac{|v^-(x)- v^-(y)||\varphi_\delta(x)-\varphi_\delta(y)|}{|x-y|^{n+2s}}dy\\
&\leq C\int_{\{v(x+y)< 0\}}- \frac{v(x+y)}{|y|^{n+2s}},
\end{align*}
which is bounded in the same fashion as in \eqref{Mv-}.
 One can be more precise with the previous bound, since $\varphi_\delta$ is smooth, then
\begin{align*}
-\cB^+(\varphi_\delta, v^-)&\leq (1-s)\Lambda^{-1}\int_{\R^N}\frac{|v^-(x)- v^-(y)||\varphi_\delta(x)-\varphi_\delta(y)|}{|x-y|^{n+2s}}dy\\
&\leq C\int_{\{v(x+y)< 0\}}- \frac{v(x+y)}{|y|^{n+2s-1}}.
\end{align*}
This is not necessary since we were able to control $\cM^-v^-$ which dominates the nonlocal drift.

Taking these bounds into account we deduce from \eqref{eqnw2}
\[
\cM^-w-\frac{2}{\varphi_\delta}\cB^-(\varphi_\delta,w)-M_1|Dw|\leq \frac{C}{\delta}(\theta r)^{-n-2s}, \ \ \text{in } B_{\theta r/2}(x_0)
\]
Let $\alpha\in (1/4,3/4)$ to be fixed and consider
\[
w_\alpha(x_0)=(C_\theta-\alpha)\frac{u(x_0)}{\varphi_\delta(x_0)},
\]
and note that 
\[
|\{w(x)>w_\alpha(x_0)\}\cap B_{\theta r/4}(x_0)|=\left|\left\{u(x)<\alpha\frac{u(x_0)\varphi_\delta(x)}{\varphi_\delta(x_0)}\right\}\cap B_{\theta r/4} \right|.
\]
Since $\varphi_\delta$ is smooth, $C^1$ suffices, we have 
\[
|\varphi_\delta(x)-\varphi_\delta(x_0)|\leq C\theta r,
\]
and so
\[
\left|\frac{\varphi_\delta(x)}{\varphi_\delta(x_0)}-1\right|\leq C\theta r.
\]
Since $\theta$ is small, we can pick $\alpha\in (1/4,3/4)$ so that 
\[
\alpha\frac{u(x_0)\varphi_\delta(x)}{\varphi_\delta(x_0)}\geq \frac{1}{2},
\]
and hence we get the bound
\begin{align}
\label{alphacontrol}\left|\left\{u(x)<\frac{u(x_0)}{2}\right\}\cap B_{\theta r/4} \right|&\leq\left|\left\{u(x)<\alpha\frac{u(x_0)\varphi_\delta(x)}{\varphi_\delta(x_0)}\right\}\cap B_{\theta r/4} \right|\\
\nonumber &=|\{w(x)>w_\alpha(x_0)\}\cap B_{\theta r/4}(x_0)|.
\end{align}
We bound $|\{w(x)>w_\alpha(x_0)\}\cap B_{\theta r/4}(x_0)|$ by below using the $L^\varepsilon$ lemma (with $\zeta=-\varphi_\delta$, $\kappa=1/\varphi_\delta$) to obtain
\begin{align*}
|\{w(x)>w_\alpha(x_0)\}\cap B_{\theta r/4}(x_0)|&\leq C (\theta r)^n\left(w(x_0)+\frac{C(\theta r)^{-n -2s}}{\delta}(\theta r)^{2s}\right)^\varepsilon  w_\alpha(x_0)^{-\varepsilon}\\
&\leq C(\theta r)^n\left(\frac{C_\theta-1}{C_\theta-\alpha}+\frac{1}{\delta}(\theta r)^{-n}w_\alpha(x_0)^{-1}\right)^\varepsilon\\
&\leq C(\theta r)^n\left(C_\theta-1+\frac{1}{\delta}(\theta r)^{-n}w_\alpha(x_0)^{-1}\right)^\varepsilon \\
&\quad\ \ \text{since }\alpha\in\left(\frac{1}{4},\frac{3}{4}\right)\\
&\leq C(\theta r)^n\left((C_\theta-1)^\varepsilon+\frac{1}{\delta}(\theta r)^{-n\varepsilon}t^{-\varepsilon}\right).
\end{align*}
The last inequality comes from the fact that 
\[
t=d^{\gamma}\frac{u(x_0)}{\varphi_\delta(x_0)}
\]
and so using the definition of $w_\alpha(x_0)$ we get
\[
t=\frac{d^\gamma}{C_\theta-\alpha}w_\alpha(x_0).
\]
Let us choose now $\theta>0$ independent of $t$ so that
\[
C(\theta r)^n(C_\theta-1)^\varepsilon\leq \frac{1}{4}|B_{\theta r/4}(x_0)|.
\] 
Since $\delta$ is fixed and independent of $t$, $u$ we have that for large values of $t$, 
\[
C(\theta r)^n\frac{1}{\delta}(\theta r)^{-n\varepsilon}t^{-\varepsilon}\leq\frac{1}{4}|B_{\theta r/4}(x_0)|.
\]
Combining this estimate and the bound in \eqref{alphacontrol} we deduce
\[
\left|\left\{u(x)<\frac{u(x_0)}{2}\right\}\cap B_{\theta r/4} \right|\leq \frac{1}{2}|B_{\theta r/4}(x_0)|, 
\]
which for $t$ large implies
\begin{align}\label{contradiction2}
\left|\left\{u(x)>\frac{u(x_0)}{2}\right\}\cap B_{\theta r/4} \right|\geq c B_{\theta r}(x_0)|.
\end{align}
Note that inequality \eqref{contradiction2} contradicts \eqref{contradiction1}, hence $t$ is uniformly bounded, which concludes the proof.
\end{proof}


\section{Proof of Theorem~\ref{teoex}.}\label{SecVP}


Let $\Omega\subset\R^N$ a general domain, which is not necessarily bounded or smooth. We will prove that there exists a generalized eigenpair $(\phi_1^+, \lambda_1^+)$ solving the eigenvalue problem
\begin{equation}\label{vpvp}
\mathcal I \phi_1^++b(x)\cdot D \phi_1^+ =  \lambda_1^+  \phi_1^+  \quad \mbox{in} \ \Omega; \quad \phi_1^+ = 0 \quad \mbox{in} \ \Omega^c.
\end{equation}

Moreover, $\lambda_1^+ = \lambda_1^+(\Omega)$ and can be characterized as \eqref{car}, that is,
\begin{equation}\label{cara1}
\lambda_1^+ = \sup \{ \lambda \ | \ \exists \ \phi \geq 0 \ \mbox{in} \ \R^N, \ \phi>0  \ \mbox{in} \ \Omega, \ \mbox{s.t.} \ \mathcal I \phi+b \cdot D \phi \leq -\lambda \phi  \ \mbox{in} \ \Omega \}.
\end{equation}
Observe that for general domains simplicity and/or uniqueness might not be true.

For $R > 0$  denote $w_{s,R}(y) = (R + |y|)^{-(N + 2s)}$ and we omit the subscript $R$ when $R = 1$. Also denote
$$
\| u \|_{L^1(w_{s, R})} = \int_{\R^N} u(y) w_{s,R}(y)dy.
$$

We need the following preliminary lemmas.

\begin{lema}\label{lema1}
Let $b \in L^\infty_{loc}(B_{2R})$. Assume that $u$ is a bounded, nonnegative viscosity solution of
\begin{equation*}
\mathcal I(u) + b(x) \cdot Du \leq c_1 u \quad \mbox{in} \ B_{2R}.
\end{equation*}
Then, there exists a constant $C > 0$ such that
\begin{equation*}
\| u \|_{L^1(w_{s, R})} \leq C\Big{(} R^{-1} ||b||_{L^\infty(B_{2R})} + ||u||_{L^\infty(B_{2R})} \Big{)} \inf_{B_R} u
\end{equation*}
\end{lema}
\begin{proof}
We assume $R =1$ and conclude the general result by considering the usual rescaling $u_R(x) = u(Rx)$. We follow the ideas of~\cite{BGQ}. 

Let $\varphi \in C_0^\infty(B_{3/2})$ such that $0 \leq \varphi \leq 1$ and $\varphi \equiv 1$ in $B_1$. If $u$ is nontrivial, the strong maximum principle implies that $u > 0$ in $B_2$. Then, there exists $0 < t \leq \inf_{B_1} u$ such that $u \leq t\varphi$ in $\R^N$. Moreover, enlarging $t$ if necessary, we can consider a point $z_0 \in B_{3/2}$ for which $u(z_0) = t\varphi(z_0)$. Then, the viscosity inequality for $u$ allows us to write
\begin{equation*}
t \I[B_{1/4}(z_0)](\varphi, z_0) + \I[B_{1/4}(z_0)^c](u,z_0) - C t \|b\|_\infty \leq C t \varphi(z_0),
\end{equation*}
where $C > 0$ just depends on $c_1$ and universal constants. From here, by the smoothness of $\varphi$ it is direct to see the existence of a universal constants $C > 0$ just depending on $N, s$ such that
\begin{equation*}
\gamma \int \limits_{B_{1/4}(z_0)^c} \frac{u(y)}{|z_0 - y|^{N + 2s}}dy \leq C \Gamma \|u \|_{L^\infty(B_2)} + C t.
\end{equation*}
The result follows by rearranging terms and using the definition of $t$.
\end{proof}

Next lemma states a known fact that solutions of equation of the type \eqref{vpvp} are H\"older continuous and localizes the $L^\infty$ dependence of the right hand side.
\begin{lema}\label{lema3}
Let $\Theta \subset \R^N$ be a domain and assume $B_{R}\subset \Theta$. Let $u$ be a solution of 
\[
\mathcal I u+b \cdot Du =  \lambda  u \quad \mbox{in} \ \Theta.
\]

Then, there exists $\alpha \in (0,1)$ and a constant $C=C_R$ not depending on the domain such that
\begin{equation*}
[u]_{C^\alpha(B_{R/2})}  \leq C_R(1 + \| u\|_{L^\infty(B_R)}  + \| u\|_{L^1(w_s)}).
\end{equation*}
\end{lema}

\begin{proof}
Without loss of generality we can assume that $R=1$ and then conclude by scaling. Note first that $u$ satisfies (in the viscosity sense)
\begin{align*}
 \M^+ u - C_1 |Du| &\leq C_2,\\
\M^- u + C_1 |Du| &\geq -C_2,
\end{align*} 
in $B_1$, for some constants $C_1, C_2 > 0$. It is direct to check that the function $\tilde u := u(x) \mathbf{1}_{B_4}(x)$, $x \in \R^N$, satisfies the inequalities
\begin{align*}
& \M^+ \tilde u - C_1 |D \tilde u| \leq C_2  + C \| u \|_{L^1(w_s)} \quad \mbox{in} \ B_1 \\
& \M^- \tilde u + C_1 |D \tilde u| \geq -C_2 \quad \mbox{in} \ B_1,
\end{align*} 
where the constant $C > 0$ in the right-hand side of the first inequality just depends on $N, s$ and  the ellipticity constants. Standard regularity theory asserts the existence of $\alpha \in (0,1)$ and a constant $C > 0$ such that
\begin{equation*}
[u]_{C^\alpha(B_{1/2})} = [\tilde u]_{C^\alpha(B_{1/2})} \leq C(1+ \| u\|_{L^\infty(B_1)}  + \| u \|_{L^1(w_s)})
\end{equation*}
\end{proof}

\begin{proof} [Proof of Theorem \ref{teoex}]Following \cite{Hamel}, we can find a sequence of smooth domains $\Omega_n$ such that 
$$\bar \Omega_n\subset \Omega_{n+1}\quad\mbox{ and }\quad \cup_{n\in\N}  \Omega_n=\Omega$$
Let now $(\phi_n, \lambda_n)$ be the eigenpair associated to the problem 
\begin{equation}\label{eigenn}
\I  (\bar u,x) +b(x)\cdot D \bar u(x) = -\lambda_n \bar u \quad \mbox{in} \ \Omega_n.
\end{equation}
Since $\Omega_n$ is smooth and bounded we have that  $\lambda_n$ is a decreasing sequence and 
$\lambda_n\geq\lambda_1^+ $ (see \cite{DQT}) and so $\lambda_n$ converges to $\bar \lambda \geq \lambda_1^+$.  

Fix $x_0\in \Omega_0$ and consider the normalization $\phi_n(x_0)=1$. By the Harnack inequality (Theorem \ref{Teoharnack}) we get, for each fixed $\bar\Omega_k$, that 

$$\sup_{\bar\Omega_k} \phi_n \leq C_\k.$$
Using the uniform bound in compact sets for the $L^\infty $ norm of the sequence and Lemma \ref{lema3} we get that there exists a subsequence $\phi_n$ which converges in $C^\alpha(\bar\Omega_k)$ for $k$ fixed, that is  $\phi_n\to \phi$ in $\bar\Omega_k$. Now, by standard stability results of viscosity solutions we get $\I \phi + b \cdot D\phi =-\bar \lambda \phi$ in  $\bar\Omega_k$. Finally, by a diagonal argument we can find a subsequence $\phi_n\to \phi$ in  $C^\alpha_{loc}(\Omega)$, which by stability of viscosity solution satisfies $\I  \phi + b \cdot D\phi=-\bar \lambda \phi$ in $\Omega$. Notice that $\phi\geq 0$ in $(\Omega)^c$ (that can be empty). Then the strong maximum principle yields $\phi > 0$ in $\Omega$ and so $\phi$ can be used in the characterization \eqref{car} to get $\lambda_1^+\geq  \bar \lambda.$ 
\end{proof}

\section{Barriers}\label{SecBarriers}

%

In this section we construct sub and supersolutions to the eigenvalue problem \eqref{eqteo} in $\R^N\setminus\{0\}$. The estimates are based on the computations of Lemma 2.1 in \cite{BV}, the fundamental solution of $\mathcal M^+$ and a rescaling argument.

\subsection{Subsolutions.} For $\beta, M > 1$ define
\begin{equation*}
\varphi(x) = \varphi_{M, \beta}(x) := (M^2 + |x|^2)^{-\beta/2}\quad \mbox{in} \quad\R^N.
\end{equation*}

We have the following estimate.
\begin{lema}\label{lemavarphi}
	Let $\beta \geq N + 2s$. There exist $\lambda > 1$ and $C, c > 0$ just depending on $N, s$ and the ellipticity constants, such that, for each $M \geq 1$, the function $\varphi = \varphi_{M, \beta}$ satisfies
	\begin{equation*}
\mathcal M^-\varphi(x) \geq  \left \{ \begin{array}{cl} -C \Gamma M^{-2s} \varphi(x) \quad & \mbox{if} \ x \in B_{\lambda M}, \\ c \gamma M^{-2s} \varphi(x) \quad & \mbox{if} \ x \in B_{\lambda M}^c. \end{array} \right .
	\end{equation*}
\end{lema}

\begin{proof}
Let $\lambda > 1$ to be fixed.  We first deal with the case $|x| \geq \lambda M > 1$. 
	
Let $K \in \mathcal K_0$ and write $L = L_K$ its associated linear operator. We have
$$
	L[B_{|x|/2}](\varphi, x) = \int_{B_{|x|/2}}\varphi(y)K(x - y)dy - \varphi(x)\int_{B_{|x|/2}} K(x-y)dy, 
$$	
and from here we have
	\begin{align*}
	L[B_{|x|/2}](\varphi, x)\geq & \gamma (2/3)^{N + 2s}|x|^{-(N + 2s)} \int_{B_{|x|/2}} \varphi(y)dy -\Gamma C \varphi(x) \int_{|x|/2}^{+\infty} r^{-1 - 2s}dr \\
	\geq & c \gamma M^{-\beta} |x|^{-\beta}\int_{B_{\lambda M/2}} (1 + |y/M|^2)^{-\beta/2}dy  - C \Gamma |x|^{-2s}\varphi(x),
	\end{align*}
	where in the last inequality we have used that $\beta \geq N + 2s$ and $|x| \geq 1$.
	
	Now, since $\lambda > 1$, we see that
	\begin{equation*}
	\int_{B_{\lambda M/2}} (1 + |y/M|^2)^{-\beta/2}dy = M^N \int_{B_{\lambda/2}} (1 + |z|^2)^{-\beta/2} dz \geq c M^N,
	\end{equation*}
	for some $c > 0$ independent of $\lambda$. 
	From here, recalling that $|x| \geq \lambda M$, we can write
	\begin{align*}
	L[B_{|x|/2}](\varphi, x) \geq & \Big{(} c \gamma M^{-2s} |x|^{-\beta} (M^2 + |x|^{2})^{\beta/2} - C \Gamma |x|^{-2s} \Big{)} \varphi(x)\\
	\geq & (c\gamma - C \Gamma \lambda^{-2s}) M^{-2s} \varphi(x),
	\end{align*}
	and therefore,  taking $\lambda$ large enough just in terms of $N, s$ and the ellipticity constants, we arrive at
	\begin{equation}\label{sub2}
	L_K[B_{|x|/2}](\varphi, x) \geq c \gamma M^{-2s} \varphi(x),
	\end{equation}
for some universal constant $c > 0$, for all $K \in \mathcal K_0$.	
	
	
	\medskip
	
	Now we deal with the integral over $B_{|x|/2}(x)$. Note first that
	\begin{equation*}
	D^2 \varphi(x) = \beta (\beta + 2) (M^2 + |x|^2)^{-\beta/2 - 2} x \otimes x - \beta (M^2 + |x|^2)^{-\beta/2 - 1} I_N.
	\end{equation*}
	We perform a Taylor expansion and disregard the positive term to write
	\begin{align*}
	& L[B_{|x|/2}(x)] (\varphi, x) \\
	\geq & \frac{-\beta \Gamma}{2} \int \limits_{B_{|x|/2}(0)} \int_{0}^{1} (1 - t) (M^2 + |x + tz|^2)^{-\beta/2 - 1} |z|^{2 - N - 2s}dt dz  \\
	\geq &  -C \Gamma (M^2 + |x|^2)^{-\beta/2 - 1} \int \limits_{B_{|x|/2}(0)}  |z|^{2 - N - 2s} dz \\
	\geq &  -C \Gamma (M^2 + |x|^2)^{-\beta/2 - 1} |x|^{2 - 2s} \\
	\geq & -C \Gamma \frac{|x|^2}{M^2 + |x|^2} |x|^{- 2s} \varphi(x),
	\end{align*}
	and from here, recalling that $|x| \geq \lambda M$, we conclude that
	\begin{equation}\label{sub3}
	L[B_{|x|/2}(x)] (\varphi, x)  \geq -C \Gamma \lambda^{-2s} M^{-2s} \varphi(x).
	\end{equation}
	
	
	Finally, for the integral term in $\tilde B := (B_{|x|/2} \cup B_{|x|/2}(x))^c$, we use that $\varphi\geq 0$ to conclude
	\begin{align*}
	L[\tilde B] (\varphi, x) \geq -\varphi(x) \Gamma \int_{\tilde B} \frac{dy}{|x - y|^{N + 2s}} \geq -C |x|^{-2s} \varphi(x),
	\end{align*}
	and then we get
	\begin{equation}\label{sub4}
	L [\tilde B] (\varphi, x) \geq -C \Gamma \lambda^{-2s} M^{-2s} \varphi(x).
	\end{equation}
	
	At this point, we collect estimates~\eqref{sub2}-\eqref{sub4} to conclude the existence of $c, C > 0$ such that, that for all $\lambda > 1$ large enough, and all $M > 1$ we have
	\begin{equation*}
	L_K(\varphi, x) \geq M^{-2s} \Big{(} c   -C\lambda^{-2s} \Big{)} \varphi(x), \quad \mbox{for} \ |x| \geq \lambda M,
	\end{equation*}
	for all $K \in \mathcal K_0$. Now we fix $\lambda$ such that $C\lambda^{-2s} \leq c/2$ (again, just depending on $N$, $s$ and the ellipticity constant) to conclude that
	\begin{equation}\label{Lvarphi1}
	\mathcal M^- (\varphi, x) \geq c_0 M^{-2s}\varphi \quad \mbox{in} \ B^c_{\lambda M},
	\end{equation}
	for some $c_0 > 0$ just depending on $N, s$ and the ellipticity constants.
	
	\medskip
	
	Now we deal with the case $|x| \leq \lambda M$. Notice that by rescaling, denoting $\tilde \varphi_1(x) = \varphi_1(x/M)$, we have for all $x \in \R^N$ that
	\begin{align*}
	L_K (\varphi_M, x) 
	= & M^{-\beta} L_K (\tilde \varphi_1, x) 
	= M^{-(\beta + 2s)} L_K (\varphi_1, x/M) \geq  -C\Gamma M^{-(\beta + 2s)},
	\end{align*}
	for some universal constant $C > 0$. The last inequality is due to the fact that $\varphi_1$ has uniform $C^2$ estimates in $\R^N$.
	
	Thus, for $|x| \leq \lambda M$ we get that
	\begin{equation*}
	L_K \varphi(x) \geq -C \Gamma M^{-2s} (1 + \lambda^2)^{\beta/2} \varphi(x), 
	\end{equation*}
	for some $C > 0$ just depending on $N$, $s$. Then, since $\lambda$ is already fixed, there exists $C > 0$ such that 
	\begin{equation}\label{sub5}
	\M^- \varphi \geq -C\Gamma M^{-2s} \varphi \quad \mbox{in} \ B_{\lambda M}.
	\end{equation}
	
	Then, joining~\eqref{Lvarphi1} and~\eqref{sub5} we conclude the result.
	\end{proof}

\medskip

As a corollary, we have
\begin{cor}\label{corvarphi}
Let $\beta \geq N + 2s$. There exists $M_0 > 1$ large enough just in terms of $N, s$ and the ellipticity constants such that, for all $M \geq M_0$, there exists $c= c_M \in (0,1)$ such that the function $\varphi$ satisfies
	\begin{equation*}
	\mathcal M^-(\varphi)+\frac{1}{2s} x\cdot D\varphi \geq -\frac{\beta - c}{2s} \varphi \quad \mbox{in} \ \R^N.
	\end{equation*}
\end{cor}

\begin{proof}
	It is direct to check that for all $x \in \R^N$ we have
	\begin{equation*}
	\frac{1}{2s} x \cdot D\varphi(x) = \frac{-\beta}{2s} \frac{|x|^2}{M^2 + |x|^2} \varphi(x).
	\end{equation*}
	
	Then, if $|x| \geq \lambda M$ we have 
	\begin{equation}\label{sub1}
	\frac{1}{2s} x \cdot D\varphi(x) \geq -\frac{\beta}{2s} \varphi(x),
	\end{equation}
	and using the estimates of the previous lemma, we conclude the estimate asserted in the statement of the corollary for $|x| \geq \lambda M$.
	
	On the other hand, if $|x| \leq \lambda M$, we see that
	\begin{equation*}
	\frac{1}{2s} x \cdot D\varphi(x) \geq \frac{-\beta}{2s} \frac{\lambda^2}{1 + \lambda^2} \varphi(x),
	\end{equation*}
	and therefore, since $\lambda > 1$ is already fixed universal, there exists a constant $c \in (0,1)$ just depending on $N$ and $s$ so that
	\begin{equation*}
	\frac{1}{2s} x \cdot D\varphi(x) \geq \frac{-(\beta - c)}{2s} \varphi(x), \quad x \in B_{\lambda M}.
	\end{equation*}
	
	The previous estimate and~\eqref{sub5} allows us to take $M$ large enough to conclude that 
	\begin{equation*}
	\mathcal M^+(\varphi)+\frac{x}{2s}\cdot D\varphi \geq -\frac{\beta - c}{2s} \varphi \quad \mbox{in} \ B_{\lambda M},
	\end{equation*}
	for some $c > 0$ small enough. 
	%
\end{proof}

\subsection{Supersolution.}\label{sec-sup}  In~\cite{FQ}, the authors proved that there exists $\tilde N$, with $-N < \sigma:=-\tilde N + 2s < 0$ so that the function
\begin{equation}\label{fundamental}
E(x) := |x|^{\sigma} \quad x \in\R^N\setminus\{0\},
\end{equation}
solves 
$$
\M^+ E(x) = 0 \quad x \in\R^N\setminus\{0\},
$$
that is, $E$ is a fundamental solution for the extremal operator $\M^+$. Given $c, \beta > 0$, we denote
\begin{equation*}
\Phi(x) = \Phi_{\beta, c}(x) = \min \{ cE(x), |x|^{-\beta} \}\quad \mbox{in} \quad\R^N\setminus\{0\}.
\end{equation*}  

\begin{lema}\label{lemasuper}
	Let $N < \beta \leq N + 2s$ and for $c > 0$, denote $r_c = c^{\frac{-1}{\beta + \sigma}}$.	
	There exists $C, c_0 > 0$ just depending on $N$ and $s$ such that, for all $c \in (0, c_0)$, the function $\Phi$ satisfies the following inequality in the viscosity  sense
	\begin{equation}\label{lemasuper1}
	\mathcal M^+\Phi \leq \left \{ \begin{array}{cl} 0  \quad & \mbox{if} \ 0 < |x| \leq r_c, \\ C \Gamma c^{\theta} \Phi(x) \quad & \mbox{if} \ |x| > r_c, \end{array} \right .
	\end{equation}
	with $\theta = \frac{\beta - N}{\beta + \sigma} > 0$.
	
\end{lema}

\begin{proof}
	Since $N < \beta$, we have $\beta + \sigma > 0$ and therefore $r_c \to +\infty$ if $c \to 0$.
	Notice that by definition of $r_c$, we have
	$$
	\Phi(x) = cE(x) = |x|^\beta \quad \mbox{if} \ |x| = r_c.
	$$ 
	
	We immediately see that at if $|x| = r_c$, no test function touching from below to $\Phi$ at $x$ exists, and therefore the viscosity inequality holds.
	
	\medskip
	
	For $0<|x| < r_c$, we notice that 
	$
	\Phi(x) = cE(x).
	$ 
	For any $K \in \mathcal K_0$ we have
	\begin{align*}
	L_K \Phi(x) = & \mathrm{P.V.} \int_{B_{r_c}} (cE(y) - cE(x)) K(x - y)dy + \int_{B_{r_c}^c} (|y|^{-\beta} - cE(x)) K(x - y)dy \\
	= & c L_K E(x) + \int_{B_{r_c}^c}(|y|^{-\beta} - cE(y)) K(x - y)dy \\
	\leq & c \M^+ E(x),
	\end{align*}
	from which we conclude that
	\begin{equation}\label{Phi0}
	\M^+\Phi(x) \leq  0 \quad \mbox{for} \ 0 < |x| < r_c.	
	\end{equation}
	
	\medskip

	
	Now we deal with the case $|x| > r_c$. Notice that in this case we have $\Phi(x) = |x|^\beta$.
	
	Given $K \in \mathcal K_0$, we write
	\begin{equation*}
	L_K \Phi(x) = I_1 + I_2 + I_3,
	\end{equation*}
	with 
	\begin{align*}
	I_1 = & L_K[B_{|x|/2}(x)] \Phi(x), \\
 I_2 = & L_K[B_{r_c} \setminus B_{|x|/2}(x)] \Phi(x), \\
	I_3 = & L_K[(B_{r_c} \cup B_{|x|/2}(x))^c] \Phi(x).
	\end{align*}
	
	For $I_1$, we use that
	$$
	I_1 \leq \mathrm{P.V.} \int_{B_{|x|/2}(x)} [|y|^\beta - |x|^\beta] K(x - y)dy.
	$$
	
	From here, we perform a Taylor expansion and proceed as the computations leading inequality \eqref{sub3} to get
	\begin{align*}
	I_1 \leq  & \frac{\Gamma \beta}{2} \int \limits_{B_{|x|/2}(0)} \int \limits_{0}^{1} (1 - t) 
	|x + tz|^{-(\beta + 2)} \Big{(} (\beta + 2)  \langle \widehat{x + tz}, z\rangle^2 - |z|^2 \Big{)}|z|^{-(N + 2s)}dz \\
	\leq & C \Gamma  |x|^{-(\beta + 2)}\int_{B_{|x|/2}} |z|^{2 - N - 2s}dz,
	\end{align*}
	for some $C > 0$ just depending on $N,s$. Then, we get 
\begin{equation*}
I_1 \leq C \Gamma |x|^{-\beta} |x|^{-2s} = C \Gamma |x|^{-2s} \Phi(x) \leq C \Gamma r_c^{-2s} \Phi(x), 
\end{equation*}
and we conclude that 
\begin{equation}\label{Phi1}
I_1 \leq C \Gamma c^{\frac{2s}{\beta + \sigma}} \Phi(x),
\end{equation}
for some $C > 0$ just depending on $N, s$.

\medskip

For $I_2$, using that $\beta \leq N + 2s$ we can write
\begin{align*}
I_2 \leq &\Gamma \int_{B_{r_c} \setminus B_{|x|/2}(x)} \frac{cE(y)}{|x - y|^{N + 2s}}dy \\
\leq & \Gamma2^\beta |x|^{-(N + 2s)} \int_{B_{r_c} \setminus B_{|x|/2}(x)} cE(y)dy \\
\leq & C \Gamma c |x|^{-\beta} \int_{0}^{r_c} t^{\sigma + N - 1}dt,
\end{align*}
for some constant $C > 0$ just depending on $N, s$. 

Thus, recalling that $r_c = c^{\frac{-1}{\beta + \sigma}}$, we arrive at
\begin{equation*}
I_2 \leq C c r_c^{\sigma + N} \Phi(x) = C \Gamma c^{1 - \frac{\sigma + N}{\beta  + \sigma}} \Phi(x),
\end{equation*}
from which we conclude that
\begin{equation}\label{Phi2}
I_2 \leq C \Gamma c^{\frac{\beta - N}{\beta + \sigma}} \Phi(x), 
\end{equation}
for some $C > 0$ just depending on $N,s$.

\medskip

Finally, for $I_3$ we see that
\begin{align*}
I_3 \leq & \Gamma \int_{(B_{r_c} \cup B_{|x|/2}(x))^c} \frac{|y|^{-\beta}}{|x - y|^{N + 2s}}dy \\
\leq & \Gamma 2^\beta|x|^{-(N + 2s)} \int_{(B_{r_c} \cup B_{|x|/2}(x))^c} |y|^{-\beta}dy \\
\leq & C \Gamma |x|^{-\beta} \int_{r_c}^{+\infty} t^{-\beta}t^{N - 1}dt,
\end{align*}
for some constant $C > 0$. A direct computation leads to
\begin{equation*}
I_3 \leq C \Gamma r_c^{-\beta + N} \Phi(x) = C \Gamma c^{\frac{\beta - N}{\beta + \sigma}} \Phi(x).
\end{equation*}
which is the same estimate as~\eqref{Phi2}, possibly relabeling $C$. Thus, collecting the estimates for $I_1, I_2$ and $I_3$, and taking supremum on $K$, we conclude that
\begin{equation*}
\M^+ \Phi(x) \leq C \Gamma c^{\frac{\beta - N}{\beta + \sigma}} \Phi(x) \quad \mbox{for} \ |x| > r_c.
\end{equation*}
This concludes the proof.
\end{proof}

%
%
%

\section{Proof of Theorem~\ref{mainTheo}}\label{SecDecay}


This section is entirely dedicated to the proof of Theorem~\ref{mainTheo}. We present the proof in the case $\M = \M^+$, the case $\M = \M^-$ is analogous. 

Consider the sequence $\{ (\phi_n, \lambda_n) \}_n$, where the pair solves the eigenvalue problem in the ball $B_n$ for $n \geq 2$, with $\phi_n > 0$ in $B_n$. We immediately remark that the sequence of eigenvalues $\lambda_n$ is decreasing and we have the bounds
\begin{equation*}
0 < \lambda_n \leq \lambda_1^+(B_1 \setminus B_{1/2}) < +\infty, \quad \mbox{for all} \ n.
\end{equation*}

Then, by Theorem~\ref{teoex} we have 
\begin{equation}\label{limlambdan}
\lim \limits_{n \to \infty} \lambda_n = \lambda_1^+(\R^N).
\end{equation}

Moreover, the family $(\phi_n)$ is uniformly bounded and H\"older continuous in each compact set of $\R^N$ by Lemmas~\ref{lema1} and~\ref{lema3}. Thus, normalized as $\phi_n(0) = 1$ for all $n$ we have it converges locally uniformly to the nontrivial eigenfunction $\phi_1^+$ solving the problem in $\R^N$.

\medskip
Now we divide the proof in several steps.

\medskip
\noindent
\textsl{1.- Lower bound for $\lambda_1^+(\R^N)$:} 
Let $n$ be fixed and consider $\epsilon \in (0,1)$ small. Denote $(\phi_\epsilon, \lambda_\epsilon)$ the positive eigenpair associated to the eigenvalue problem in the set $B_n \setminus B_\epsilon$ (we omit the dependence on $n$, but stress on its dependence in the estimates). 

Considering the fundamental solution $E$ given by~\eqref{fundamental}, we have 
\begin{equation*}
\M^+(E_+) + \frac{1}{2s} x \cdot DE_+(x) = \frac{\sigma}{2s} E_+(x), \quad x \in B_\epsilon^c,
\end{equation*}
and from the characterization of the principal eigenvalue, we conclude that 
$$
\lambda_\epsilon \geq -\sigma/2s,
$$ 
for all $\epsilon  > 0$, and the sequence $(\lambda_\epsilon)$ is decreasing in $\epsilon$.


Fix $x_0 \in B_n \setminus B_1$ and normalize the family as $\phi_\epsilon(x_0) = 1$. By the Harnack inequality in Theorem~\ref{Teoharnack}, we have that the family of solutions $(\phi_\epsilon)$ is uniformly bounded in each compact set contained in  $B_\epsilon^c$.  In particular, for $\rho > 0$ to be fixed, and for all $\epsilon$ small enough, there exists $R > 0$ depending on $n$ and $\rho$ such that 
$$
0 \leq \phi_\epsilon \leq R \quad \mbox{in} \ B_\rho^c. 
$$

Consider the function $U(x) = R (\mathbf 1_{B_\rho}(x) + 1$) and notice that for each $K \in \mathcal K_0$ and each $x \in B_\rho$ we have
\begin{equation*}
L_K(U,x) = -R \int \limits_{B_\rho^c} K(x - y)dy \leq -\gamma R \int \limits_{B_\rho^c} \frac{dy}{|x - y|^{N + 2s}} \leq -\gamma R \int \limits_{B_\rho^c} |z|^{-(N + 2s)}dz,
\end{equation*}
and from here we get that 
\begin{equation*}
\M^+ (u,x) \leq -c R \gamma \rho^{-2s} \quad \mbox{in} \ B_\rho.
\end{equation*}

Thus, if we denote $\lambda_0 = \lambda_1^+(B_1 \setminus B_{1/2})$, there exists $\rho > 0$ small enough so that
\begin{equation}\label{usach}
\M^+(U) + \frac{1}{2s} x \cdot DU + \lambda_0 U \leq 0 \quad \mbox{in} \ B_\rho,
\end{equation}
and clearly $U \geq \phi_\epsilon$ in $B_\rho^c$. Then, we can fix $\rho$ small enough so that inequality holds~\eqref{usach} and the associated operator 
satisfies the comparison principle in $B_\rho$, see~\cite{DQT1}. Since $\lambda_n \leq \lambda_0$, then we have 
$$
\M^+(\phi_\epsilon) + \frac{1}{2s} x \cdot D\phi_\epsilon + \lambda_0 \phi_\epsilon \geq 0 \quad \mbox{in} \ B_\rho \setminus B_\epsilon
$$ 
for all $\epsilon$ small. Since $\phi_\epsilon \leq U$ in $B_\epsilon \cup B_n^c$, by the comparison principle we conclude that $\phi_\epsilon \leq 2R$ in $\R^N$, from which $(\phi_\epsilon)_\epsilon$ is uniformly bounded for $n$ fixed.

Thus, by the $C^\alpha$ estimates and stability of viscosity solutions, we conclude $\lambda_\epsilon \to \lambda$, $\phi_\epsilon \to \phi$ in $C^\alpha_{loc}(\R^N \setminus \{ 0\})$ as $\epsilon \to 0$, where $\phi$ solves the equation
\begin{equation}\label{eqhoyo}
\M^+ (\phi) + \frac{1}{2s} x \cdot D\phi = -\lambda \phi \quad \mbox{in} \ B_n \setminus \{ 0\}; \qquad \phi = 0 \quad \mbox{in} \ B_n^c.
\end{equation}

Consider the extension of $\phi$ at the origin, that we still denote by $\phi$, given by
$$
\phi(0) = \liminf_{x \to 0, |x| > 0} \phi(x),
$$
which is lower semicontinuous function. 

\medskip

{\textit{Claim:}} $\phi$ is a viscosity supersolution to problem~\eqref{eqhoyo} in $B_n$. 

Assume the claim holds. Using the claim and the characterization of the principal eigenvalue, we conclude that $\lambda_1^+(B_n) \geq -\sigma/2s$ for all $n \in \N$. This last fact, together with~\eqref{limlambdan} allows us to conclude
\begin{equation}\label{lower}
\lambda_1^+(\R^N) \geq -\sigma/2s,
\end{equation} 
which is the desired estimate of this step. 

\medskip

{\textit{Proof of Claim:}} Let $\varphi$ be a smooth bounded function such that $\phi - \varphi$ attains its strict global minimum at $x = 0$. Consider the fundamental solution $E$ in~\eqref{fundamental} and for $\beta > 0$ small enough, consider the function 
\begin{equation*}
x \mapsto \phi(x) - \varphi(x) + \beta E(x), \quad x \neq 0.
\end{equation*}

It is direct to see that this function attains its global minimum at a point $x_\beta \neq 0$ for all $\beta$ small enough. Moreover, by the minimality of $x_\beta$ and the positivity of $E_+$, for all $z \neq 0$ we have
\begin{equation*}
\phi(x_\beta) - \varphi(x_\beta) \leq \phi(z) - \varphi(z) + \beta E_+(z),
\end{equation*}
and from here, taking liminf in $\beta$ in the last inequality for $z$ fixed, and using that the origin is the strict minimum of $\phi - \varphi$, we conclude that $x_\beta \to 0$ as $\beta \to 0$. Now, observe that for $z \neq 0$ fixed, the last inequality together with the lower semicontinuity of $\phi$ and the continuity of $\phi$ away of the origin, allows us to write
\begin{align*}
(\phi - \varphi)(0) & \leq \liminf_{\beta \to 0 }(\phi - \varphi)(x_\beta) \\
& \leq \limsup_{\beta \to 0 }(\phi - \varphi)(x_\beta) \\
& \leq \limsup_{\beta \to 0 }(\phi - \varphi + \beta E)(z) \\
& \leq (\phi - \varphi)(z).
\end{align*}

Now, by definition, there exists a sequence $x_k \to 0, x_k \neq 0$ such that $\phi(x_k) \to \phi(0)$. Taking $z = x_k$ in the above inequalities and making $k \to \infty$ we conclude that $\phi(x_\beta) \to \phi(0)$ by the continuity of $\varphi$. 

Then, we use $\varphi - \beta E$ as test function for $\phi$ at $x_\beta$ to conclude that
\begin{equation*}
\M_+(\varphi - \beta E_+, x_\beta) + \frac{1}{2s}  x_\beta \cdot \Big{(} D\varphi(x_\beta) - \beta D E_+(x_\beta )\Big{)} \leq - \lambda \phi(x_\beta).
\end{equation*}

Notice that since $\sigma < 0$ we have $x_\beta \cdot DE_+(x_\beta) \leq 0$, and using well-known properties of maximal operators we arrive at
\begin{equation*}
\M^+(\varphi, x_\beta) - \M^+(E, x_\beta) + \frac{1}{2s} x_\beta \cdot D\varphi(x_\beta) \leq -  \lambda \phi(x_\beta).
\end{equation*}
Since $E$ is the fundamental solution and $x_\beta \to 0$, $\phi(x_\beta) \to \phi(0)$ as $\beta \to 0$, we can take the limit in $\beta$ to conclude the viscosity inequality for $\phi$ at the origin, by the smoothness of $\varphi$. This finishes the claim, and therefore the estimate~\eqref{lower}.



\medskip
\noindent
\textsl{2.- Lower bound for the decay of $\phi_1^+$:} Given the sequence $(\phi_n, \lambda_n)$ described at the beginning of the proof, normalized as $\phi_n(0) = 1$, we conclude by the Harnack inequality that $\phi_n$ converges locally uniformly to the solution $(\phi_1^+, \lambda_1^+)$ of the eigenvalue problem~\eqref{eqteo} with the same normalization. By the strong maximum principle, we know that $\phi_1^+ > 0$ in $\R^N$. 

We claim that there exists a constant $c > 0$ such that
\begin{equation}\label{lowdecay}
\phi_1^+(x) \geq c |x|^{-(N + 2s)} \quad \mbox{for all} \ x \in B_1^c.
\end{equation}

A key ingredient is the construction of an appropriate subsolution. Let us consider a nonnegative function $\eta \in C^\infty(\R^N)$ with support in the unit ball $B_1$ and $\| \eta \|_{L^1} =1$. Recalling the exponent $\sigma$ in~\eqref{fundamental}, let $\sigma' \in (\sigma, 0)$ very close to $\sigma$, and denote $\tilde E(x) = |x|^{\sigma'}$. Denote $\beta = N + 2s$. In view of the definition of $\sigma'$ we can assume that the inequality $-\beta < \sigma'$ holds. 

For $C_0, M > 1$ and $\epsilon > 0$ to be fixed, we define
\begin{equation*}
\psi(x) = (M + |x|^2)^{-\beta/2} + C_0 \eta(x) - \epsilon \tilde E(x), \quad x \neq 0.
\end{equation*}

For $|x| > 2$, using the computations in Lemma~\ref{lemavarphi} and the estimates in~\cite{FQ}, we can write
\begin{align*}
\M^+ (\psi,x) \geq & \M^- (\varphi_M, x) + C_0\M^+(\eta, x) + \M^-(-\epsilon \tilde E, x) \\
\geq & -C \Gamma M^{-\beta} \varphi_M(x) + C_0 \Gamma \int_{B_1}  \frac{\eta(y)}{|x - y|^{N + 2s}}dy - \epsilon \M^+(\tilde E, x)\\
\geq &  -C \Gamma M^{-\beta} \varphi_M(x) + c \Gamma C_0  |x|^{-\beta} \int_{B_1} \eta(y) dy - \epsilon c(\sigma') |x|^{\sigma' - 2s}\\
\geq &  -C \Gamma M^{-\beta} \varphi_M(x) + c \Gamma C_0  \varphi_M(x) - \epsilon c(\sigma') |x|^{\sigma' - 2s},
\end{align*}
for some constants $C, c > 0$ just depending on $N, s$ and $c(\sigma') < 0$. From here, taking $M$ large and $C_0$ large enough, we conclude that 
\begin{equation*}
\M^+ (\psi) \geq c C_0 \varphi_M \quad \mbox{in} \ B_2^c,
\end{equation*}
for some constant $c > 0$ just depending on $N, s$ and the ellipticity constants. 

A direct computation leads us to
\begin{equation*}
x \cdot D\psi(x) = -\beta \frac{|x|^2}{M^2 + |x|^2} \varphi_M(x) - \sigma' \epsilon \tilde E(x) \geq -\beta \varphi_M(x) - \sigma' \epsilon \tilde E(x),
\end{equation*}
for all $|x| > 0$. Then, by the above estimate for $\M^+ (\psi)$, taking $C_0$ large enough just in terms of $N, s$ and the ellipticity constants, we conclude that
\begin{equation}\label{psi1}
\M^+ \psi(x) + \frac{1}{2s} x \cdot D\psi(x) \geq -\frac{\sigma'}{2s} \epsilon \tilde E(x) \geq \frac{\sigma'}{2s} \psi(x), \quad x \in B_2^c.
\end{equation}

From here, we fix $C_0, M$ in order the above inequality holds. 

\medskip

Now we come back to the proof of~\eqref{lowdecay}. By contradiction, we assume the existence of a sequence $c_k \to 0^+$ and $x_k \in \R^N$ with $|x_k| \to +\infty$ such that
\begin{equation}\label{contra1}
\phi_1^+(x_k) < c_k |x_k|^{-(N + 2s)}, \quad \mbox{for all $k \in \N$ large enough}.
\end{equation} 

We know that $\phi_1^+ > 0$ in $\R^N$. Thus, multiplying $\psi$ by a small constant, we can assume that $\phi_1^+ > \psi$ in $B_2$. In addition, since $\sigma' < \beta$, there exists $R = R_\epsilon$ such that $\psi \leq 0$ in $B_R^c$, and $R \to +\infty$ as $\epsilon \to 0$. In fact, noticing that $\psi \to \varphi_M$ in $B_2^c$ as $\epsilon \to 0$, by~\eqref{contra1}, we can take $\epsilon = \epsilon_k$ small enough to get $\psi(x_k) > \phi_1^+(x_k)$ for some $k$ large. 

Hence, there exists $t_k > 1$ such that the following holds: there exists $\tilde x_k \in B_2^c$ such that $\psi(\tilde x_k) = t_k \phi_1^+(\tilde x_k)$ and $\psi \leq \phi_1^+$ in $\R^N \setminus \{ 0 \}$. Then, we can use $\psi$ as test function for $t_k \phi_1^+$ at $\tilde x_k$ and write 

%
%

%
\begin{equation*}
\M^+ \psi(\tilde x_k) + \frac{1}{2s} \tilde x_k \cdot D\psi(\tilde x_k) \leq -\lambda_1^+ t_k \phi_1^+(\tilde x_k) = -\lambda_1^+ \psi(\tilde x_k), 
\end{equation*}
but using~\eqref{psi1} we have
\begin{equation*}
\M^+ \psi(\tilde x_k) + \frac{1}{2s} \tilde x_k \cdot D\psi(\tilde x_k) \geq \frac{\sigma'}{2s} \psi(\tilde x_k),
\end{equation*}
from which we conclude that $\lambda_1^+ \leq \frac{-\sigma'}{2s} < \frac{-\sigma}{2s}$, and this contradicts~\eqref{lower}.

\medskip
\noindent
\textsl{3.- Upper bound for $\lambda_1^+(\R^N)$:} Here we prove that 
\begin{equation}\label{upperlambda}
\lambda_1^+(\R^N) < \frac{N + 2s - \delta}{2s},
\end{equation}
for some $\delta > 0$ small enough.

Let $\delta' > 0$ small. Consider 
$$
N + 2s < \beta \leq N + 2s + \delta'.
$$ 

Let $\varphi$ as in Corollary~\ref{corvarphi} with $\beta$ as above. 


Since
$$
\lim_{|x|\to +\infty}\frac{\phi_1^+}{\varphi}=\infty,
$$
there exists $\eta > 0$ and a point $x_\eta \in \R^N$ such that $\eta \varphi(x_\eta) = \phi_1^+(x_\eta)$ and $\eta \varphi \leq \phi_1^+$ in $\R^N$. Thus, we can use $\eta \varphi$ as test function to write
\begin{equation*}
\M^+ \eta \varphi(x_\eta) + \frac{1}{2s} x_\eta \cdot D \eta \varphi(x_\eta)
 \leq -\lambda_1^+ \eta \varphi(x_\eta),
\end{equation*}
but using the estimate given in Corollary~\ref{corvarphi}, we conclude that
\begin{equation*}
-\frac{\beta - \delta}{2s} \eta \varphi(x_\eta) \leq -\lambda_1^+ \eta \varphi(x_\eta),
\end{equation*}
for some $\delta > 0$ small enough. Then, $\lambda_1^+ \leq (\beta - \delta)/2s$ and taking $\delta' \to 0$ we conclude the result. 


\medskip
\noindent
\textsl{4.- Upper bound for the decay of $\phi_1^+$:} We are going to prove that there exists $C > 0$ such that
\begin{equation*}
\phi_1^+(x) \leq C|x|^{-(N + 2s)} \quad \mbox{for} \ |x|>1.
\end{equation*}

By contradiction, assume there exist $C_k \to \infty$ and $x_k$ such that 
$$
\phi_1^+(x_k) > C_k |x_k|^{-(N + 2s)}
$$ 
for all $k$ large. Note that necessarily, we have $|x_k| \to \infty$.

Consider the approximating sequence $(\phi_n, \lambda_n)$. Then, there exists a sequence $n_k \to \infty$ such that $\phi_{n_k}(x_k) \geq C_k |x_k|^{-(N + 2s)}$. 

Now, let $\delta > 0$ as in the previous estimate and consider $\beta = N + 2s$. Let $\Phi$ as in Lemma~\ref{lemasuper} and take $c > 0$ very small in terms of $\delta$ to conclude that
\begin{equation}\label{super}
\M^+ \Phi  + \frac{1}{2s} x \cdot D \Phi \leq -\frac{\beta - \delta/2}{2s} \Phi \quad \mbox{in} \ B_{r_c}^c.
\end{equation}


At this point we fix $c > 0$ such that~\eqref{super} holds, and take $k$ large enough in order to have $|x_k| > r_c$. Thus, $\Phi(x_k) = |x_k|^{-\beta}$.

Since each $\phi_{n_k}$ is compactly supported, enlarging $C_k$ if necessary, we can assume that for all $k$ there exists a sequence $\tilde x_k$ with $|\tilde x_k| \to +\infty$, such that $\phi_{n_k}^+(\tilde x_k) = C_k \Phi(\tilde x_k)$ and $\phi_{n_k} \leq C_k \Phi$ in $\R^N \setminus \{ 0 \}$.

Then, we use $C_k \Phi$ as a test function for $\phi_{n_k}$ at $\tilde x_k$, from which
\begin{equation*}
\M^+ C_k \Phi(\tilde x_k)  + \frac{1}{2s} \tilde x_k \cdot D C_k \Phi(\tilde x_k)
 \geq -\lambda_{n_k} C_k \Phi(\tilde x_k).
\end{equation*}

From~\eqref{super}, we obtain
\begin{equation*}
\lambda_{n_k} \geq \frac{\beta - \delta/2}{2s}.
\end{equation*}

Taking limit as $k \to \infty$, we arrive at
\begin{equation*}
\lambda_1^+(\R^N) \geq \frac{\beta - \delta/2}{2s},
\end{equation*}
which contradicts~\eqref{upperlambda}. 



\medskip
\noindent \textsl{5.-Simplicity.} In this last step, we require the following version of maximum principle, which is the nonlocal version of Lemma 3.1 in \cite{AT}.


\begin{lema}\label{aaa} Let $\alpha < \frac{N + 2s}{2s}$.
Let $u$  be a bounded viscosity solution of 
$$
\M^+ u + \frac{1}{2s} x \cdot Du \geq -\alpha u\quad\mbox{in}\quad B_{R}, 
$$
with $u\leq C_0|x|^{-(N + 2s)}$ for some $C_0 > 0$, and $u\leq 0$ in $B_R$. Then there exists $R_0$ so that if $R>R_0$, then $u\leq 0$ in $\R^N$. 
\end{lema}

\begin{proof}The result for $\alpha \leq 0$ is direct. Assume $0 < \alpha$ and let $\beta$ such that
$2s \alpha < \beta < N + 2s$ and denote $\epsilon = \beta - 2s\alpha$. Consider $\Phi$ the function of Lemma~\ref{lemasuper} defined with such an exponent $\beta$. 

Take $R_0 > r_c$ as in Lemma~\ref{lemasuper} and fix $c$ small enough in order to have
\begin{equation*}
\M+ \Phi + \frac{1}{2s} x \cdot \Phi \leq -\frac{\beta - \epsilon/2}{2s} \Phi \quad \mbox{in} \ B_{r_c}^c.
\end{equation*}

Suppose now by contradiction that for some $R \geq R_0$, there exists $x_R$ with $|x_R|>R$ such that $u(x_R)>0$. Then, since $u \leq 0$ in $B_R$ and $u$ decays faster than $\Phi$ at infinity, there exists $\eta > 0$ such that $u \leq \eta \Phi$ in $\R^N \setminus \{ 0 \}$ and a point $x \in B_{R}^c$ such that $u(x) = \eta \Phi(x)$. 

Then, using $\eta \Phi$ as a test function for $u$ and the estimate above, we obtain
\begin{equation*}
\alpha \geq \frac{\beta - \epsilon/2}{2s},
\end{equation*}
but this is a contradiction.
\end{proof}

Now note that by step 3 above $\alpha= \lambda_1^+$ satisfies the hypothesis of the Lemma \ref{aaa}.

\medskip
\noindent
{\textit{Proof of simplicity of $\lambda_1^+$:}} Suppose there exists another positive eigenfunction $\tilde\phi$ with eigenvalue $\tilde\lambda$. By the characterization of $\lambda_1^+$, $\tilde\lambda\leq \lambda_1^+$ and we also have the decay estimate of step 4 for $\tilde\phi$.  We normalize $\tilde \phi$ in such a way
$ \tilde\phi(x_0)>\phi_1^+(x_0)$, for $x_0$ fixed. Then, 
$$
\M^+ \tilde \phi + \frac{1}{2s} x \cdot D \tilde \phi =-\tilde\lambda\tilde\phi\geq -\lambda_1^+\tilde\phi,
$$
and $w_s:=\tilde\phi-s\phi_1^+$ is negative in $B_{R_0}$ for $s$ large. Note that 
$$
\M^+ \tilde w_s + \frac{1}{2s} x \cdot D w_s
\geq -\lambda_1^+w_s.
$$ 

Then, by the Lemma \ref{aaa} and the strong maximum principle we have $w_s<0$ in  $\R^N$.

Now define 
\[
s^*=\inf\{s>1\, |\,w_s< 0\,\quad\mbox{in}\quad\R^N \}.
\]
Notice that $s^*>1$ by the normalization, since $w_{s^*}\leq 0$. If $w_{s^*}\equiv 0$ there is nothing to prove.
Now, if there is a point $\tilde x$ such that $w_{s^*}(\tilde x)<0$, by the strong maximum principle $w_{s^*}<0$ in $\R^N $. Let $\delta>0$ small such that $w_{s^*-\delta}<0$ in $B_{R_0}$. By Lemma \ref{aaa} and the strong maximum principle $w_{s^*-\delta}<0$ in $\R^N$, which contradicts the definition of $s^*$. This concludes the proof of simplicity.

Finally, since the equation is invariant under rotation and by simplicity of the eigenvalue, we deduce that $\phi_1^+$ is radially symmetric. 

The proof of Theorem~\ref{mainTheo} is now complete.
\qed

\bigskip

\noindent {\bf Acknowledgements.} 

G. D. was partially supported by Fondecyt Grant No. 1190209.

A. Q. was partially supported by Fondecyt Grant No. 1190282  and Programa Basal, CMM. U. de Chile. 

E. T. was partially supported by  Conicyt PIA Grant No. 79150056, and Fondecyt Iniciaci\'on No. 11160817.


\end{document}